\documentclass[12pt]{amsart}
\usepackage{amssymb,amsmath,amsthm,amscd}
\usepackage[all]{xy}

\numberwithin{equation}{section}

\newtheoremstyle{fancy1}{10pt}{10pt}{\itshape}{12pt}{\textsc\bgroup}{.\egroup}{8pt}{
}
\newtheoremstyle{fancy2}{10pt}{10pt}{}{12pt}{\itshape}{.}{8pt}{ }

\theoremstyle{fancy1}

\newtheorem{cor}[equation]{Corollary}
\newtheorem*{cora*}{Corollary A}
\newtheorem*{corb*}{Corollary B}
\newtheorem*{corc*}{Corollary C}
\newtheorem*{cord*}{Corollary D}
\newtheorem*{core*}{Corollary E}
\newtheorem{lem}[equation]{Lemma}
\newtheorem{prop}[equation]{Proposition}
\newtheorem{thm}[equation]{Theorem}

\newtheorem*{main*}{Result}
\newtheorem*{conjecture*}{Conjecture}
\newtheorem*{cor*}{Corollary}

\newcommand{\gal}[2]{\operatorname{Gal}\left( #1 / #2 \right)}

\newtheorem*{def*}{Definition}
\newtheorem{rem}[equation]{Remark}
\newtheorem*{rem*}{Remark}

\newtheorem*{example*}{Example}

\newtheorem*{examples*}{Examples}

\newtheorem*{prop*}{Proposition}
\theoremstyle{remark}
\newtheorem*{case*}{Case}

\newtheorem*{proofA*}{Proof of Corollary A}
\newtheorem*{proofB*}{Proof of Corollary B}
\newtheorem*{proofC*}{Proof of Corollary C}
\newtheorem*{proofD*}{Proof of Corollary D}
\newtheorem*{proofE*}{Proof of Corollary E}
\newtheorem*{proofprop*}{Proof of Proposition}
\newtheorem*{proofthm*}{Proof of Theorem \ref{maintheorem}}

\newcommand{\cref}[1]{Corollary~\ref{#1}}

\newcommand{\F}{{\mathbb{F}}}

\newcommand{\Z}{{\mathbb{Z}}}

\def\con#1=#2(#3){#1 \equiv #2 \bmod{#3}}

\newcommand{\Aut}{\ensuremath{\operatorname{Aut}}}

\newcommand{\GL}{\ensuremath{\operatorname{GL}}}

\begin{document}

\title[Infinite class towers for function fields]{Infinite class towers for function fields}

\author{Jing Long Hoelscher}
\address{University of Illinois at Chicago\\
     Chicago, IL 60607}
\email{jlong@math.uic.edu}

\begin{abstract}
This paper gives examples of function fields $K_0$ over a finite field $\mathbb{F}_q$ of $p$ power order ramified only at one finite regular prime over $\mathbb{F}_q(t)$, which admit infinite Hilbert $p$-class field towers. Such a $K_0$ can be taken as an extension of a cyclotomic function field $\mathbb{F}_q(t)(\lambda_{\mathfrak{p}^m})$ for a certain regular prime $\mathfrak{p}$ in $\mathbb{F}_q[t]$.

\end{abstract}

\maketitle


\section{Introduction}
Let $K=\mathbb{F}_q(t)$ be a function field of one variable over the finite field $\mathbb{F}_q$ of $q$ elements, where $q$ is a power of a prime $p$. Let $A=\mathbb{F}_q[t]$ be the ring of integers in $K$, integral away from $\infty$ in $K$. Let $K_0$ be a finite separable extension of $K$. Denote by $S_{K_0}$ the set of all infinite places in $K_0$. For $i=0,1,2,\dots$, take the maximal abelian unramified $p$-extension $K_{i+1}$ of $K_i$ in which all infinite places split completely. If there are no integers $i\geq 0$ such that $K_i=K_{i+1}$, we say that $K_0$ \emph{admits an infinite Hilbert $(p,S_{K_0})$-class field tower}. For any multiplicative abelian group $A$, denote by $d_p(A)=\dim_{\mathbb{F}_p}(A/A^p)$ the $p$-rank of $A$. For a function field $F$, denote by $\mathcal{O}_F$ the ring of integers of $F$ integral away from all infinite places and denote by $\mathcal{O}^*_F$ the group of units in $\mathcal{O}_F$. Also denote by $Cl_F$ and $\tilde{Cl}_F$ the ideal class groups of $F$ and $\mathcal{O}_F$ respectively. Schoof showed in \cite{Sc} that $K_0$ admits an infinite Hilbert $(p,S_{K_0})$-class field tower if $d_p(\tilde{Cl}_{K_0})$ is big enough, i.e.\ $$d_p(\tilde{Cl}_{K_0})\geq2+2\sqrt{d_p(\mathcal{O}^*_{K_0})+1},$$ or if the number $\rho$ of ramified primes in $K_0/K_{-1}$, an intermediate cyclic extension of degree $p$ of $K_0/K$, is big enough, i.e.\ 
$$\rho\geq 3+d_p(\mathcal{O}_{K_{-1}}^*)+2\sqrt{d_p(\mathcal{O}_{K_0}^*)+1},$$ 
given that the set $S_{K_0}$ is stable under $\gal{K_0}{K_{-1}}$. This paper will give examples of extensions $K_0$ of cyclotomic function fields $K(\lambda_{\mathfrak{p}^m})$ that admit infinite Hilbert $(p,S_{K_0})$-class field towers, where $d_p(\tilde{Cl}_{K(\lambda_{\mathfrak{p}^m})})=0$ and there is only one finite prime ramified in $K_0/K$, whereas $d_p(\mathcal{O}_{K_0}^*)$ is very big.
 
Let $K^{ac}$ be an algebraic closure of $K$. For any $f\in A$ and $u\in K^{ac}$, Carlitz defined in \cite{Ca1} and \cite{Ca2} an action $u^f=f(\varphi+\mu)(u)$, where $\varphi: K^{ac}\rightarrow K^{ac}$ is the Frobenius automorphism $u\rightarrow u^q$ and $\mu: K^{ac}\rightarrow  K^{ac}$ is multiplication by $t$. This action gives $K^{ac}$ an $A$-module structure. The set of $f$-torsion points $\Lambda_f:=\{u\in K^{ac}| u^f=0\}$ is a cyclic $A$-module. Denote by $\lambda_f$ a generator of $\Lambda_f$. The function field $K(\lambda_f)$ obtained by adding all $f$-torsion points is the cyclotomic function field for $f$. An irreducible $\mathfrak{p}\in A$ is defined to be regular if $p\nmid |Cl_{K(\lambda_{\mathfrak{p}})}|$ and irregular otherwise. If $\mathfrak{p}$ is very irregular, thus $d_p(\tilde{Cl}_{K(\lambda_{\mathfrak{p}^m})})$ is big, then Schoof's theorem shows that $K(\lambda_{\mathfrak{p}^m})$ will be very likely to admit an infinite Hilbert class tower. The examples $K_0$ in this paper, which have infinite Hilbert class field towers, will be extensions of cyclotomic function fields $K(\lambda_{\mathfrak{p}^m})$ with certain regular irreducibles $\mathfrak{p}$ and positive integers $m$. The main theorem is:

\begin{thm}\label{maintheorem}
Let $\mathfrak{p}$ be an irreducible in $\F_q[t]$. Suppose there exist a cyclotomic function field $K(\lambda_{\mathfrak{p}^m})=\F_q(t)(\lambda_{\mathfrak{p}^m})$ for some $m\in\mathbb{N}$ and a cyclic unramified Galois extension $H/K(\lambda_{\mathfrak{p}^m})$ of prime degree $h\neq p$, in which all infinite places split completely, which satisfy either of the following two conditions:
\begin{enumerate}
\item[(I)]$p\mid |\tilde{Cl}_H|$, $p\nmid |\tilde{Cl}_{K(\lambda_{\mathfrak{p}^m})}|$  and $f_{p,h}^2-4f_{p,h}\geq \frac{4h\cdot\varphi(\mathfrak{p}^m)}{q-1}$;
\item[(II)]$p\nmid |\tilde{Cl}_H|$ and $h\geq \frac{4p\cdot\varphi(\mathfrak{p}^m)}{q-1}+4$,
\end{enumerate}
where $f_{p,h}$ is the order of $p$ in $(\Z/h\Z)^*$ and $\varphi(\mathfrak{p}^m)=q^{dm}-q^{d(m-1)}$ with $d=\deg(\mathfrak{p})$. Then there is a function field $K_0$ ramified only at $\mathfrak{p}$ and $\infty$ over $\F_q(t)$ that admits an infinite Hilbert $(p,S_{K_0})$-class field tower.
\end{thm}

As an application we will show:
\begin{thm}\label{infiniteclassfield2}
There exist extensions $K_0$ over cyclotomic function fields $K(\lambda_{\mathfrak{p}^m})$, ramified over $\F_q(t)$ only at one regular finite prime $\mathfrak{p}\in\mathbb{F}_q[t]$, which admits an infinite $(p,S_{K_0})$-class field tower.
\end{thm}
\bigskip

\section{Infinite class field towers}

Let $K(\lambda_{\mathfrak{p}^m})=\F_q(t)(\lambda_{\mathfrak{p}^m})$, and let $H$ be a cyclic unramified Galois extension over $K(\lambda_{\mathfrak{p}^m})$ of prime degree $h$, where all infinite places split completely. Denote by $\mathcal{O}_{K(\lambda_{\mathfrak{p}^m})}$ and $\mathcal{O}_H$ the rings of integers of $K(\lambda_{\mathfrak{p}^m})$ and $H$, integral away from the infinite places. Also denote by $\tilde{Cl}_{K(\lambda_{\mathfrak{p}^m})}$ and $\tilde{Cl}_H$ the ideal class groups of $\mathcal{O}_{K(\lambda_{\mathfrak{p}^m})}$ and $\mathcal{O}_H$.

For any profinite $p$-group $G$, let $h_i=d_p( H^i(G,\Z/p\Z))$ for all $i>0$. Recall \v{S}afarevi\v{c}-Golod's theorem (Theorem 2.1 of \cite{Sc}), which says that if $G$ is a non-trivial finite $p$-group, then $h_2>\frac{h_1^2}{4}$. We will use the \v{S}afarevi\v{c}-Golod theorem for the proofs of Propositions $\ref{pdividclassnumber}$ and $\ref{pnotdividclassnumber}$. 

\begin{prop}\label{pdividclassnumber}
Under condition (I) in Theorem \ref{maintheorem}, the function field $H$ has an infinite Hilbert $(p,S_H)$-class field tower, where $S_H$ is the set of infinite places in $H$.
\end{prop}
\begin{proof}
Since $p||\tilde{Cl}_H|$, we can pick a degree $p$ cyclic unramified extension $M_0/H$ in which all infinite places split completely. Condition (I) says $p\nmid |\tilde{Cl}_{K(\lambda_{\mathfrak{p}^m})}|$. So $(p,h)=1$. Take the Galois closure $\bar{M}_0$ of $M_0/K(\lambda_{\mathfrak{p}^m})$, which is the composite of all conjugates of $M_0$ over $K(\lambda_{\mathfrak{p}^m})$. The Galois group $\gal{\bar{M}_0}{H}$ is of the form $(\Z/p\Z)^l$ for some $l\in\mathbb{N}$, and $\gal{\bar{M}_0}{K(\lambda_{\mathfrak{p}^m})}\cong (\Z/p\Z)^l\rtimes\Z/h\Z$. The semi-direct product gives a non-trivial homomorphism $\Z/h\Z\rightarrow \GL_l(\mathbb{F}_p)\cong \Aut((\Z/p\Z)^l)$, so $h\,|\,\hspace{0.1em}|\GL_l(\mathbb{F}_p)|=(p^l-1)(p^l-p)\cdots(p^l-p^{l-1})$. Then $h\,|\,(p^i-1)$, for some $i\leq l$, i.e.\, $p^i=1$ mod $h$, and the order $f_{(p,h)}$ of $p$ in $(\Z/h\Z)^*$ divides $i$. Thus $f_{p,h}\leq i\leq l$. Also there are $l$ linearly disjoint cyclic degree $p$ unramified extensions of $H$, so $h_1=d_p(H^1(G,\Z/p\Z))\geq l$, where $G=\gal{\Omega_H}{H}$ is the Galois group of the maximal unramified $p$-extension $\Omega_H$ of $H$ where all infinite places split completely. Therefore $f_{p,h}\leq i\leq l \leq h_1$. Since $f_{p,h}>4$ by condition (I) and the function $x^2-4x$ is increasing on $x\geq 2$, we have $f_{p,h}^2-4f_{p,h}\leq h_1^2-4h_1$. If $\Omega_H$ were finite, the \v{S}afarevi\v{c}-Golod theorem would imply that $h_1^2-4h_1<4h_2-4h_1$. Also by the proof in Theorem 2.3 in \cite{Sc} we would have $h_2-h_1\leq |S_H|-1$. Thus $$f_{p,h}^2-4f_{p,h}\leq h_1^2-4h_1<4h_2-4h_1< 4|S_H|=\frac{4h\cdot\varphi(\mathfrak{p}^m)}{q-1},$$ in contradiction to condition (I).  
\end{proof}

{From} now on we will focus on the case under condition (II) in Theorem \ref{maintheorem}. We know the principal prime ideal $(\lambda_{\mathfrak{p}^m})$ in $K(\lambda_{\mathfrak{p}^m})$ above $\mathfrak{p}$ splits completely as $(\lambda_{\mathfrak{p}^m})=\prod_{i=1}^{h} \mathfrak{p}_i $ in the subfield $H$ of the Hilbert class field of $K(\lambda_{\mathfrak{p}^m})$. 
\begin{prop}\label{rayclass}
Assume $p\nmid |Cl_H|$, then for any $\mathfrak{p}_i$ in $H$ above $\mathfrak{p}$ there exist integers $k$ such that $p| |\tilde{Cl}_H^{\mathfrak{p}_i^k}|$, where $\tilde{Cl}_H^{\mathfrak{p}_i^k}$ is the ray class group of $\mathcal{O}_H$ for the modulus $\mathfrak{p}_i^k$.
\end{prop}

\begin{proof}For any integer $k\in \mathbb{N}$, we have the following exact sequence from class field theory (Proposition 1.1 in \cite{Au})
$$ \mathcal{O}_{S_H}^* \longrightarrow U_{\mathfrak{p}_i}/U_{\mathfrak{p}_i}^{(k)}\longrightarrow \tilde{Cl}_H^{\mathfrak{p}_i^k}\longrightarrow \tilde{Cl}_H\longrightarrow 1,   \hspace{2em} (*)$$
where $\mathcal{O}_{S_H}^*$ is the group of units outside $S_H$, the set of all infinite places in $H$. Consider the $\mathfrak{p}$-adic completion $\hat{K}_{\mathfrak{p}}$ of $K=\F_q(t)$ at $\mathfrak{p}$ and the $\mathfrak{p}_i$-adic completion $\hat{H}_{\mathfrak{p}_i}\cong \hat{K}_{\mathfrak{p}}(\lambda_{\mathfrak{p}^m})$ of $H$ at $\mathfrak{p}_i$. Denote by $\mathcal{O}_{\hat{H}_{\mathfrak{p}_i}}$ and $\hat{\mathfrak{p}}_i$ the valuation ring in $\hat{H}_{\mathfrak{p}_i}$ and its maximal ideal. Denote by $U_{\mathfrak{p}_i}=\{u\in \hat{H}_{\mathfrak{p}_i}| v_{\mathfrak{p}_i}(u)=0\}$ the unit group in $\mathcal{O}_{\hat{H}_{\mathfrak{p}_i}}$ and denote by $U_{\mathfrak{p}_i}^{(k)}=1+\hat{\mathfrak{p}}_i^k$ the $k$th one-unit group. The kernel of the first map is $\mathcal{O}_{S_H}^{\mathfrak{p}_i^k}=\{u\in \mathcal{O}_{S_H}^*| u\equiv 1$ mod $\mathfrak{p}_i^k \}$. Now take the projective limit of the exact sequence (*), 
$$1\rightarrow \lim_{\leftarrow}\mathcal{O}_{S_H}^*/\mathcal{O}_{S_H}^{\mathfrak{p}_i^k}\rightarrow \lim_{\leftarrow}U_{\mathfrak{p}_i}/U_{\mathfrak{p}_i}^{(k)}\rightarrow \lim_{\leftarrow} \tilde{Cl}^{\mathfrak{p}_i^k}_H\rightarrow \tilde{Cl}_H\rightarrow 1.  \hspace{2em} (**)$$
Now $\lim_{\leftarrow}U_{\mathfrak{p}_i}/U_{\mathfrak{p}_i}^{(k)}\cong U_{\mathfrak{p}_i}$. Proposition II.5.7 in \cite{Ne} says $U_{\mathfrak{p}_i}\cong \Z/(q-1)\Z\oplus \Z_p^{\mathbb{N}}.$ On the other hand, $\lim_{\leftarrow} \mathcal{O}_{S_H}^*/\mathcal{O}_{S_H}^{\mathfrak{p}_i^k}\cong\lim_{\leftarrow}\mathcal{O}_{S_H}^*/\mathcal{O}_{S_H}^{\mathfrak{p}_i}\times\lim_{\leftarrow}\mathcal{O}_{S_H}^{\mathfrak{p}_i}/\mathcal{O}_{S_H}^{\mathfrak{p}_i^k}$. We have a nature inclusion $i:\mathcal{O}_{S_H}^*/\mathcal{O}_{S_H}^{\mathfrak{p}_i}\hookrightarrow(\mathcal{O}_H/\mathfrak{p}_i)^*\cong \F_q^*$. As for $\lim_{\leftarrow}\mathcal{O}_{S_H}^{\mathfrak{p}_i}/\mathcal{O}_{S_H}^{\mathfrak{p}_i^k}$, it is the pro-$p$-completion of $\mathcal{O}_{S_H}^{\mathfrak{p}_i}$. The $\Z_p$-rank of $\lim_{\leftarrow}\mathcal{O}_{S_H}^{\mathfrak{p}_i}/\mathcal{O}_{S_H}^{\mathfrak{p}_i^k}$ is at most the number $|S_H|-1$ of generators of $\mathcal{O}_{S_H}^*$, which is $\frac{h\cdot\varphi(\mathfrak{p}^m)}{q-1}$. Also a strong form of Leopoldt's conjecture is proved by Kisilevsky in \cite{Ki} for function fields, which says the $|S_H|-1$ generators of $\mathcal{O}^*_{S_H}$ are also multiplicatively independent over $\Z_p$ in the completion. So we have $p| \frac{|\tilde{Cl}_H^{\mathfrak{p}_i^k}|}{|\tilde{Cl}_H|}$, thus $p||\tilde{Cl}_H^{\mathfrak{p}_i^k}|$ for sufficiently large $k$.
\end{proof}

By Proposition \ref{rayclass} above, we know $p\mid |\tilde{Cl}_H^{\mathfrak{p}_1^k}|$ for some integer $k$. Since $p$ does not divide the class number $|\tilde{Cl}_H|$, there exists a degree $p$ extension $H_1$ of $H$ ramified only at $\mathfrak{p}_1$, where all infinite places split completely. So $H_1/H$ is an Artin-Schreier extension and $H_1$ is of the form $H(y_1)$ where $y_1^p-y_1=x_1$ with $x_1\in H$. Since $H_1/H$ is only ramified at $\mathfrak{p}_1$, we may assume that $v_{\mathfrak{p}_1}(x_1)=-m<0$ and $v_{\mathfrak{q}}(x_1)\geq 0$ for $\mathfrak{q}\neq\mathfrak{p}_1$. The conjugates $H_i$ of $H_1$ over $K(\lambda_{\mathfrak{p}^m})$ are degree $p$ extensions of $H$ ramified only at $\mathfrak{p}_i$, and they are of the form $H_i=H(y_i)$ with $y_i^p-y_i=x_i$, where $x_i\in H$ and $v_{\mathfrak{p}_i}(x_i)=-m<0$ and $v_{\mathfrak{q}}(x_i)\geq 0$ for $\mathfrak{q}\neq \mathfrak{p}_i$. Consider the extension $L=H(y_0)$, where $y_0^p-y_0=x_1+x_2+\cdots +x_h$. So $y_0=y_1+y_2+\cdots+y_h$. 
\begin{lem}\label{unramified}
The extension $L(y_i)/L$ is unramified for $1\leq i\leq h$.
\end{lem}

\begin{proof}Let $M=H(y_1,\cdots,y_h)$. We will show $M/L$ is unramified. First we consider the discriminant of $H_i/H$, which is only ramified at $\mathfrak{p}_i$. Since $H_i=H(y_i)$ and $y_i^p-y_i=x_i$ with $v_{\mathfrak{p}_1}(x_1)=-m$, the discriminant of $H_i/H$ $\delta_{H_i/H}=(\mathfrak{p}_i)^{(p-1)(m+1)}$. Now consider the discriminant of the compositum $M/H$. For each $1\leq i \le h$, $H_i/H$ is ramified only at $\mathfrak{p}_i$, so these extensions are linearly disjoint. The extension $M/H$, being the compositum of these extensions, has discriminant $\delta_{M/H}=\prod_{i=1}^h (\delta_{H_i/H})^{p^{h-1}}=\prod_{i=1}^h (\mathfrak{p}_i)^{p^{h-1}\cdot (p-1)(m+1)}$. Next we consider the discriminant of $L/H$. We have $L=H(y_0)$ where $y_0^p-y_0=x_0$, the valuation $v_{\mathfrak{p}_i}(x_0)=v_{\mathfrak{p}_i}(x_1+x_2+\cdots+x_h)=-m$ and $v_{\mathfrak{q}}(x_0)\geq 0$ for any $\mathfrak{q}\nmid \mathfrak{p}$. So the discriminant $\delta_{L/H}$ of $L/H$ is $\prod_{i=1}^h (\mathfrak{p}_i)^{(p-1)(m+1)}.$ In the tower of extensions $M/L/H$,
$$\prod_{i=1}^h (\mathfrak{p}_i)^{p^{h-1}\cdot (p-1)(m+1)}=\delta_{M/H}=\delta_{L/H}^{p^{h-1}}\cdot N_{L/H}(\delta_{M/L})=(\prod_{i=1}^h (\mathfrak{p}_i)^{(p-1)(m+1)})^{p^{h-1}}\cdot N_{L/H}(\delta_{M/L}).$$
Comparing these two equations, we get $\delta_{M/L}=1$, i.e. $M/L$ is unramified.

\end{proof}
\begin{prop}\label{pnotdividclassnumber}
Under condition (II) in Theorem \ref{maintheorem}, the function field $L$ admits a $(p,S_L)$-infinite class tower, where $S_L$ is the set of infinite places in $L$.
\end{prop} 
\begin{proof}
Let $\Omega_L$ be the maximal unramified pro-$p$ extension of $L$ where all places in $S_L$ split completely, with Galois group $G=\gal{\Omega_L}{L}$. We claim $\Omega_L/L$ is infinite. Suppose otherwise $\Omega_L/L$ is a finite extension. Then $G$ is a finite $p$-group and $\Omega_L$ admits no cyclic unramified extensions of degree $p$, in which all infinite places split completely. Let $h_i=d_p(H^i(G,\Z/p\Z))$ for all $i>0$. By the proof of Theorem $2.3$ in \cite{Sc} we have
\begin{equation}
h_2-h_1\leq |S_L|-1<\frac{p\cdot\varphi(\mathfrak{p}^m)h}{q-1}.
\end{equation}
By Lemma \ref{unramified}, we have $h$ linearly disjoint unramified $p$-extensions $L(y_i)$ over $L$, so $h_1\geq h$. By the \v{S}afarevi\v{c}-Golod theorem for function fields (Theorem 2.1 in \cite{Sc}), we have $h_2>\frac{h_1^2}{4}$. Condition (II) says $h$ is a prime with $h(\geq \frac{4p\cdot\varphi(\mathfrak{p}^m)}{q-1}+4)>4$. Combining all the inequalities and the fact that the function $x^2-4x$ monotonically increases with $x$ if $x\geq 2$, we have 
$$\frac{h^2}{4}-h\leq \frac{h_1^2}{4}-h_1<h_2-h_1<\frac{p\cdot\varphi(\mathfrak{p}^{m})h}{q-1},$$ 
which implies $h<\frac{4p\cdot\varphi(\mathfrak{p}^m)}{q-1}+4$, contradiction.
\end{proof}
Combining Proposition \ref{pnotdividclassnumber} and Proposition \ref{pdividclassnumber}, we can conclude Theorem \ref{maintheorem} in the introduction.

\begin{rem}
Under condition (I), we take $K_0$ to be the cyclic unramified extension $H$ over $K(\lambda_{\mathfrak{p}^m})$. Under condition (II), we take $K_0$ to be the degree $p$ Artin-Schreier extension $L$ of $H$. Notice $p\nmid |\tilde{Cl}_{K(\lambda_{\mathfrak{p}^m})}|$, i.e.\ $d_p(\tilde{Cl}_{K(\lambda_{\mathfrak{p}^m})})=0$ in condition (I) and $p\nmid |\tilde{Cl}_H|$, i.e.\ $d_p(\tilde{Cl}_{H})=0$ in condition (II), so Schoof's argument cannot be applied to $K(\lambda_{\mathfrak{p}^m})$ and $H$ respectively.

\end{rem}

\bigskip
\section{Examples for small values of $\deg(\mathfrak{p})$}
As a consequence of Theorem \ref{maintheorem}, we will verify in the case of a small $m$ that an extension $K_0$ of the cyclotomic function field $K(\lambda_{\mathfrak{p}^m})$ has an infinite $(p,S_{K_0})$-class tower for a certain regular prime $\mathfrak{p}\in A$. Denote by $K(\lambda_{\mathfrak{p}^m})^+$ the maximal subfield of $K(\lambda_{\mathfrak{p}^m})$ in which $1/t$ splits completely. Denote by $h^+$ and $\tilde{h}^+$ the ideal class number $|Cl_{K(\lambda_{\mathfrak{p}^m})^+}|$ and $|\tilde{Cl}_{K(\lambda_{\mathfrak{p}^m})^+}|$, and define the relative class numbers $$h^-:=\frac{|Cl_{K(\lambda_{\mathfrak{p}^m})}|}{h^+};\hspace{3em} \tilde{h}^-:=\frac{|\tilde{Cl}_{K(\lambda_{\mathfrak{p}^m})}|}{\tilde{h}^+}.$$
Rosen showed in \cite{Ro} that $h^-=(q-1)^{|S_{K(\lambda_{\mathfrak{p}^m})}|-1}\cdot \tilde{h}^-$, where $S_{K(\lambda_{\mathfrak{p}^m})}$ is the set of infinite places in $K(\lambda_{\mathfrak{p}^m})$. So for a prime $h$ such that $\gcd(h, q-1)=1$, $h^-$ and $\tilde{h}^-$ has the same divisibility. 

We first consider the case $m=1$, i.e.\ the prime cyclotomic function field case. The values of $h$ in the following table are factors of the relative class numbers $h^-$ of cyclotomic function fields $K(\lambda_{\mathfrak{p}^m})$, which are obtained through the class number table in \cite{IS}. Also all computations of the order $f_{(p,h)}$ of $p$ in $(\Z/h\Z)^*$ are done by \cite{PARI2}.
{\setlength{\tabcolsep}{0.1cm}
\renewcommand{\arraystretch}{1.6}
\begin{table}[!h]
\begin{center}
\begin{tabular}{|c|c|c|c|}
\hline
$\mathfrak{p}$  & $h$  & $f_{p,h}$  \\
\hline \hline
$2+2t+t^4$  & $63648628175761$ & $31824314087880$  \\
\hline
$2+t+t^4$ & $532611841$ &  $88768640$  \\
\hline
\end{tabular}
\end{center}
\end{table}}

\begin{cor}\label{q=3}
The cyclotomic function field in Theorem \ref{maintheorem} can be taken as $\mathbb{F}_q(t)(\lambda_{\mathfrak{p}})$ where $q=3$ and the prime $\mathfrak{p}=2+2t+t^4$ or $2+t+t^4$. 
\end{cor}
\begin{proof}
In the case $\mathfrak{p}=2+2t+t^4$, the relative class number of $K(\lambda_{2+2t+t^4})$ is $2^{39}\cdot 17\cdot 97\cdot 63648628175761$. Let $H$ be a subfield of the Hilbert class field of $K(\lambda_{2+2t+t^4})$ with $\gal{H}{K}\cong \Z/63648628175761\Z$. When $3\mid |Cl_H|$, condition (I) in Theorem \ref{maintheorem} is satisfied since $2+2t+t^4$ is regular by the table in \cite{IS} and $f_{3,63648628175761}=31824314087880$ thus $f_{3,63648628175761}^2-4f_{3,63648628175761}\geq \frac{4\cdot 63648628175761 \cdot (3^4-1)}{2}$; when $3\nmid |Cl_H|$, condition (II) in Theorem \ref{maintheorem} is satisfied since $h=63648628175761$ and thus $h\geq \frac{4\cdot 3\cdot (3^4-1)}{3-1}=484$. 

In the case $\mathfrak{p}=2+t+t^4$, the relative class number of $K(\lambda_{2+t+t^4})$ is $2^{39}\cdot 241\cdot 641\cdot 881\cdot 532611841$. Let $H$ be a subfield of the Hilbert class field of $K(\lambda_{2+t+t^4})$ with $\gal{H}{K}\cong \Z/532611841\Z$. When $3\mid |Cl_H|$, condition (I) in Theorem \ref{maintheorem} is satisfied since $2+t+t^4$ is regular by the table in \cite{IS} and $f_{3,532611841}=88768640$ thus $f_{3,532611841}^2-4f_{3,532611841}\geq \frac{4\cdot 532611841 \cdot (3^4-1)}{2}$; when $3\nmid |Cl_H|$, condition (II) in Theorem \ref{maintheorem} is satisfied since $h=532611841\geq \frac{4\cdot 3\cdot (3^4-1)}{3-1}+4=484$. 

\end{proof}

Next we apply Theorem \ref{maintheorem} to the case $m>1$ and the degree $d=\deg(\mathfrak{p})=1$. Since $d=1$, we may assume $\mathfrak{p}=t$. The following values of $h$ are factors of the relative class numbers $h^-$ of cyclotomic function fields $K(\lambda_{\mathfrak{p}^m})$, which are from \cite{GS}.
{\setlength{\tabcolsep}{0.1cm}
\renewcommand{\arraystretch}{1.6}
\begin{table}[!h]
\begin{center}
\begin{tabular}{|c|c|c|c|c|c|}
\hline
$m$ & $q$  & $h$  & $f_{p,h}$ & $(f_{p,h}^2-4f_{p,h})-(\frac{4h\cdot\varphi(\mathfrak{p}^m)}{q-1})$ & $h-(\frac{4p\cdot\varphi(\mathfrak{p}^m)}{q-1}+4)$ \\
\hline \hline
$2$ & $7$ & $118147$  & $39382$ & $1547476280$ & $11647$  \\
\hline
$2$ & $11$ & $19031$ & $19030$ & $361227416$ & $18543$ \\
\hline
$3$ & $5$ & $821$ & $410$ & $84360$ & $317$\\
\hline
$4$ & $3$ & $379$ & $378$ & $100440$ & $51$ \\
\hline
$5$ & $3$ & $5779$ & $5778$ & $31489776$ & $4803$ \\
\hline  
\end{tabular}
\end{center}
\end{table}} 
\begin{cor}\label{d=1}
The cyclotomic function field in Theorem \ref{maintheorem} can be taken as $\mathbb{F}_7(t)(\lambda_{t^2})$, $\mathbb{F}_{11}(t)(\lambda_{t^2})$, $\mathbb{F}_5(t)(\lambda_{t^3})$, $\mathbb{F}_3(t)(\lambda_{t^4})$ and $\mathbb{F}_3(t)(\lambda_{t^5})$. 
\end{cor}
\begin{proof}
From the table of the class numbers $h^+$ and $h^-$ of $\mathbb{F}_q(t)(\lambda_{t^m})$ in \cite{GS}, we know in all above cases $t$ is a regular prime, thus $p\nmid |\tilde{Cl}_{\mathbb{F}_q(t)(\lambda_{t^m})}|$. To apply Theorem \ref{maintheorem}, we just need to verify the two inequalities $f_{p,h}^2-4f_{p,h}\geq\frac{4h\cdot\varphi(\mathfrak{p}^m)}{q-1}$ and $h\geq\frac{4p\cdot\varphi(\mathfrak{p}^m)}{q-1}+4$, which are shown to be true in the table above.
\end{proof}
Corollaries \ref{q=3} and \ref{d=1} complete the proof of Theorem \ref{infiniteclassfield2}. 
\bigskip
\bigskip
\providecommand{\bysame}{\leavevmode\hbox
 to3em{\hrulefill}\thinspace}

 \end{document}